\documentclass[preprint,a4paper,12pt]{elsarticle}

\usepackage[utf8]{inputenc}
\usepackage[english]{babel}
\usepackage{slantsc} 
\usepackage[mathscr]{euscript}
\usepackage{relsize}
\usepackage{amssymb}
\usepackage{hyperref}
\usepackage{xcolor}
\usepackage{amsmath}
\usepackage{amsfonts}
\usepackage{amssymb}
\usepackage{amsthm}
\usepackage{graphicx}
\usepackage{hyperref}
\usepackage{xcolor}

\definecolor{gold}{HTML}{FDD017}
\definecolor{orange}{HTML}{FF8000}
\definecolor{deep sky blue}{HTML}{3BB9FF}
\definecolor{blue}{HTML}{0101DF}
\definecolor{red}{HTML}{DF0101}
\definecolor{co}{HTML}{2020A0}
\newcommand{\gold}[1]{{\color{gold} #1}}
\newcommand{\red}[1]{{\color{red} #1}}

\newcommand{\blue}[1]{{\color{blue} #1}}

\swapnumbers

\theoremstyle{theorem}
\newtheorem{thm}{Theorem}[section]
\newtheorem{prop}[thm]{Proposition}

\theoremstyle{definition}

\newtheorem{rem}[thm]{Remark}

\newcommand{\ssum}{{\textstyle\sum}}
\newcommand{\seq}[3]{#1_{#2},\dotsc,#1_#3}
\newcommand{\floor}[1]{\lfloor#1\rfloor}
\newcommand{\ZZ}{\mathbb{Z}}
\newcommand{\CC}{\mathbb{C}}
\newcommand{\FF}{\mathbb{F}}
\newcommand{\abs}[1]{|#1|}

\begin{document}

\begin{frontmatter}
\title{A bootstrap for the number of 
$\FF_{q^r}$-rational points\\ on a curve over $\FF_q$}

\author{S. Molina$^1$, N. Sayols$^2$, and S. Xambó-Descamps$^2$}
\address{$^1\,$Centre de Recerca Matemàtica,
\textsf{\footnotesize\upshape santiago.molina@uni.lu}\\
$^2\,$Universitat Politècnca de Catalunya,
\textsf{\footnotesize\upshape narcissb@gmail.com} \&
\textsf{\footnotesize\upshape sebastia.xambo@upc.edu}}

\begin{abstract}
In this note we present a fast algorithm
that finds for any $r$ the number $N_r$ of $\FF_{q^r}$ rational points 
on a smooth absolutely irreducible curve $C$ defined over $\FF_q$
assuming that we know $N_1,\dotsc,N_g$, where $g$ is the genus of $C$.
The proof of its validity is given in detail and its working are illustrated with
several examples. In an Appendix we list the Python function
in which we have implemented the algorithm together with
other routines used in the examples. 
\end{abstract}


\end{frontmatter}

\section*{Glossary}

In this note, \textit{curve over} $\FF_q$ means 
(unless explicitly stated otherwise) 
a smooth absolutely irreducible projective 
curve over the finite field $\FF_q$ of cardinal $q$.

The \textit{Hasse-Weil-Serre upper bound} for the number 
of $\FF_q$-rational points on a curve $C/\FF_q$ is
$N_q(g)=q+1 + g \floor{2\sqrt{q}}$
\cite{Serre-1983a,Serre-1983b,Serre-1984}.
 
A curve $C/\FF_q$  is said to be \textit{maximal} if 
the number of its $\FF_q$ points equals $N_q(g)$.

For historical aspects and background, we refer
to the excellent surveys \cite{Torres-2008, vdGeer-2015}, and the 
many references provided there. 
The general context provided by the
Weil conjectures is outlined in \cite{Hartshorne-1983},
Appendix C.

\setcounter{section}{0}
\section{Ingredients}

The \textit{zeta function} of a projective variety $X/\FF_q$ is the power series
\begin{equation}\label{eq:Z}
Z(T)=Z(X/\FF_q)= \exp\left(\ssum_{r=1}^{\infty}  N_r(X)\tfrac{T^r}{r}\right),
\end{equation}
where $N_r=N_r(X)$ denotes the number of $\FF_{q^r}$-rational points of $X$. 
This function generates all the numbers $N_r$ according to the relation
\begin{equation}\label{eq:N_r}
N_r = \tfrac{1}{(r-1)!}\tfrac{d^r}{dT^r}\log Z(T)\vert_{T=0}.
\end{equation}
The needed information about $Z(T)$ is provided by the \textit{Weil conjectures}
(see \cite{Hartshorne-1983} for the history of work on them and in particular
about their proofs). For the case of curves $C/\FF_q$, which is the one we need
in this note, they were actually proved by Weil himself
\cite{Weil-1948} and can be summarized as follows (see 
\cite{Raskin-2007} for proofs in present day algebraic geometry language):

\paragraph{Rationality}
$Z(T)=\frac{P(T)}{(1-T)(1-qT)}$, with $P(T)\in \ZZ[T]$. 

\paragraph{Functional equation}
If $g$  is the genus of $C$, $P(T)=q^g T^{2g}P(1/qT)$. 
In particular, $\deg(P)=2g$.

\paragraph{Analogue of the Riemann hypothesis}
$P(T)=\prod_{j=1}^{2g}(1-\alpha_j T)$, with $\alpha_j\in\CC$ such that
$\abs{\alpha_j}=\sqrt{q}$. 

\section{Basic algorithm}

Using (\ref{eq:N_r}), it is easy to conclude that
\begin{eqnarray}\label{eq:Nr}
N_r(C)= q^r+1-S_r,\quad S_r=\sum_{j=1}^{2g} \alpha_j^r.
\end{eqnarray}
Thus we see that knowing $N_r$ is equivalent to knowing $S_r$.

Now we can describe a procedure for computing the $N_r$ for
$r>2g$ assuming that we know $N_1,\dotsc,N_{2g}$.
Since the \textit{Newton sums} $S_r=\sum_{j=1}^{2g} \alpha_j^r$
are symmetric polynomias of the $\alpha_i$, they are polynomial
expressions in the (signed) elementary symmetric polynomials 
$c_1,\dotsc,c_{2g}$ of $\alpha_1,\dotsc,\alpha_g,\alpha_{g+1},\dotsc,\alpha_{2g}$
(in other words, $c_j=(-1)^j\sigma_j$, where $\sigma_j$ 
is the standard elementary symmetric polynomial
of degree $j$ in $\alpha_1,\dotsc,\alpha_{2g}$). 
Even though these expressions were essentially derived in the \textsc{xvii}
century (first by Girard and later by Newton), for 
convenience we include their statement
and a proof in the \ref{ap:newton-sums}.

For our purposes here, the net result is that we can 
proceed as follows:
\begin{enumerate}
\item For $j=1,\dotsc,2g$, set $S_r=q^r+1-N_r$.
\item Use the formula (\ref{eq:rec2}) to 
recursively compute $c_1,\dotsc,c_{2g}$:
\[
c_j = -(S_j+c_1S_{j-1}+\cdots+c_{j-1}S_1)/j.
\]
\item Use the formula (\ref{eq:rec1})
to successively get $S_{2g+1},\dotsc,S_{r}$.
\item Set $N_i = q^i+1-S_i$ for $i=2g+1,\dotsc,r$.
\end{enumerate}

\section{An improved algorithm}\label{sec:main}

Since $P$ has real coefficients, if $\alpha_j$ is a root,
then so is $\bar{\alpha}_j=q/\alpha_j$. The possible real roots
are $\pm\sqrt{q}$, and there is an even number of them because 
the degree of $P$ is even. In fact, there must be an even number of $-\sqrt{q}$ 
(and hence an even number of $\sqrt{q}$) as otherwise
the coefficient of $T^{2g}$ (namely $q^g$) would be negative.
This implies that we can index the roots of $P$ in such a
way that $\alpha_{2g-j+1}=\bar{\alpha}_{j}=q/\alpha_j$, $j=1,\dotsc,g$.
Therefore, $P$ has the following form:
\begin{equation}\label{eq:P}
P(T) = c_{2g} T^{2g} + c_{2g-1} T^{2g-1}+\cdots+c_{1}T+c_0,
\quad c_{2g}=q^g,\ c_0=1.
\end{equation}

\begin{prop}
\label{thm:main} We have that
$c_{g+l}=q^{l}c_{g-l}$ for $l=1,\dotsc,g$.
\end{prop}
\begin{proof}
Since $\alpha_j\mapsto q/\alpha_j$ exchanges
$\alpha_1,\dotsc,\alpha_g$ and $\alpha_{2g},\dotsc,\alpha_{g+1}$, if we set $f(T)=\prod_{j=1}^{2g}(T-\alpha_j)=c_0T^{2g}+c_1T^{2g-1}+\cdots+c_{2g-1}T+c_{2g}$,
then $T^{2g}f(q/T)$ has the same roots as $f(T)$ and therefore
$T^{2g}f(q/T)=c_{2g}f(T)=q^g f(T)$. Now the claim follows by equating
the coefficients of $T^{g+l}$ on both sides: on the right we get
$q^gc_{g-l}$ and on the left $ q^{g-l}c_{g+l}$.
\end{proof}

This proposition gives the boostrap at the root of our improved algorithm: after
computing $c_1,\dotsc,c_g$ from $N_1,\dotsc,N_g$ as in the basic algorithm, 
we automatically get $c_{g+1},\dotsc,c_{2g}$, namely
$qc_{g-1},\dotsc,q^{g-1}c_1, q^gc_{0}$, and so we have the following improved procedure:

\begin{enumerate}
\item For $j=1,\dotsc,g$, set $S_j=q^j+1-N_j$.
\item Use the formula (\ref{eq:rec2}) to
get $c_j$ for $j=1,\dotsc,g$:
\[
c_j = -(S_j+c_1S_{j-1}+\cdots+c_{j-1}S_1)/j.
\]
\item For $j=g+1,\dotsc,2g$, set $c_{j}=q^{j-g}c_{2g-j}$, get
$S_{j}$ with equation (\ref{eq:rec2}), 
\[
S_{j}=-(c_1S_{j-1}+\cdots+c_{j-1}S_1+jc_{j}),
\]
and set $N_{j}=q^{j}+1-S_{j}$.
\item For $j>2g$, use the formula (\ref{eq:rec1}) to recursively compute
the 
\[
S_j=-(c_1S_{j-1}+\cdots+c_{2g}S_{j-2g})
\]
and set $N_j = q^{j}+1-S_{j}$.
\end{enumerate}

We include the listing of our Python implementation of this procedure
in \ref{ap:python-code} (the function \textsf{XN}).

\begin{rem}\label{rem:subsequences}
Thus the infinite sequence $\{N_j(C)\}_{j\ge1}$ only depends
on $q$ and the list $[N_1,\dotsc,N_g]$. One interesting 
consequence is that given a positive integer $s$, the subsequence
$\{N_{sj}(C)\}_{j\ge1}$ must be the result of computing the long
sequence for $q^s$ and the list $[N_s,\dotsc,N_{sg}]$.
\end{rem}

\section{Elliptic curves revisited}

Over $\FF_2=\ZZ_2$ there are 32 cubic polynomials in \textit{normal form}
(cf. \cite{Husemoller-1987} or \cite{Silverman-2009} for notations and terminology)
\[
E=y^2+a_1 xy +a_3 +x^3+a_2x^2+a_4x+a_6
\]
of which precisely 16 are non-singular. For these cases, $g=1$, the HWS bound is
$m=\floor{2\sqrt{2}}=2$ and all the integers in the HWS interval
$[1,5]$ occur as $N_1(E)$ for some $E$ (this can be be checked 
with Deuring's algorithm, which is explained, and implemented, in
the subsection ``The Deuring function'' of \ref{ap:python-code}).
Now a straighforward computation yields the following distribution:
\begin{align*}
N &\quad E\\[3pt]
1\; & \quad  y^2 + y + x^3 + x + 1, \ y^2 + y + x^3 + x^2 + 1\\
2\; & \quad y^2 + x y + x^3 + x^2 + 1,\quad
            y^2 + x y + x^3 + x^2 + x,\quad\\
  &\quad y^2 + (x + 1) y + x^3 + 1,\quad y^2 + (x + 1) y + x^3 + x + 1\\
3\; & \quad y^2 + y + x^3,\quad y^2 + y + x^3 + 1\\
    & \quad y^2 + y + x^3 + x^2 + x,\quad y^2 + y + x^3 + x^2 + x + 1\\
4\; & \quad y^2 + x y + x^3 + 1,\quad y^2 + x y + x^3 + x\\
    & \quad y^2 + (x + 1) y + x^3 + x^2,\quad y^2 + (x + 1) y + x^3 + x^2 + x\\
5\; & \quad y^2 + y + x^3 + x, \quad y^2 + y + x^3 + x^2
\end{align*}

Computing the sequences of values returned by \textsf{XN}
with inputs $q=2$ and $[N_1]$, for $N_1=1,\dotsc,5$, and $k=20$
we get the following data (the top row is the maximum value
$N_q(1)$ of $\# E(\FF_q)$ supplied by ``Serre's procedure'', as
described in \ref{ap:python-code}):

\newcommand{\dg}[1]{\gold{\bf#1}}
\medskip
\begin{tabular}{|c|cccccccccc|}
\hline
$S$ & 5 	  & 9 		& 14 	   & 25 	  & 44   	 & 81 		& \blue{150} & \blue{289} & \blue{558} 	& 1089 \\
\hline\hline
$N$ & 1 	  & 5 		& 13 	   & \red{25} & 41 		 & 65 		& 113 		 & 225 		  & 481 		& 1025 \\
    & 2 	  & 8 		& \red{14} & 16 	  & 22 		 & 56 		& 142 		 & \dg{288} & 518 		&  968 \\
    & 3 	  & \red{9} &  9 	   &  9 	  & 33 		 & \red{81} & 129 		 & 225 		  & 513 		& \red{1089} \\
    & 4       & 8 		&  4 	   & 16 	  & \red{44} & 56 		& 116 		 & \dg{288} 		  & 508 		&  968 \\
    & \red{5} & 5 		&  5 	   & 25 	  & 25 		 & 65 		& \dg{145}& 225 		  & \dg{545} & 1025 \\
\hline
\end{tabular}

\renewcommand{\k}{\hspace{3pt}}
\medskip
\begin{tabular}{|@{\k}c@{\k}|@{\k}c@{\k}c@{\k}c@{\k}c@{\k}c@{\k}c@{\k}c@{\k}c@{\k}c@{\k}c@{\k}|}
\hline
$S$ 
    & \blue{2139} & 4225 & 8374 & 16641 & \blue{33131} & \blue{66049} & \blue{131797} & 263169 & \blue{525737} & 1050625 \\
\hline\hline
$N$ & 2113 & \red{4225} & 8321 & 16385 & 32513 & 65025 & 130561 & 262145 & \dg{525313} & \red{1050625} \\
    & 1982 & 4144 & \red{8374} & 16472 & 32494 & \dg{65088} & 131174 & 263144 & 525086 & 1047376 \\
    & 2049 & 3969 & 8193 & \red{16641} & 32769 & 65025 & 131073 & \red{263169} & 524289 & 1046529 \\
    & \dg{2116} & 4144 & 8012 & 16472 & \dg{33044} & \dg{65088} & 130972 & 263144 & 523492 & 1047376 \\
    & 1985 & 4225 & 8065 & 16385 & 33025 & 65025 & \dg{131585} & 262145 & 523265 & 1050625 \\
\hline
\end{tabular}

The red entries are maximal values. The blue values of row $S$
indicate that no elliptic curve of the five defined over $\FF_2$ 
achieves them, and in this case
the yellow entries indicate the curve (or two curves in two cases) 
that yield the highest value. Here we remark that the tables
agree with the conclusions in Table 1 of
\cite{Yaghoobian-Blake-1994} (page 305), except for the
$k=8$, which is classified there as maximal (provided by $E_4$),
but in the date above we see that $S=289$ and that the maximum achieved
by our five elliptic curves is $288$ (two of them, $E_2$ and $E_4$).
This means that there is an elliptic curve \textit{defined over}
$\FF_{2^8}$ that has 289 rational points, one more than the maximum
of the number of $\FF_{2^8}$-rational points for our five curves. 
In fact, since in the first six columns the maximum is achieved,
Remark \ref{rem:subsequences} tells us that
that curve cannot be defined over $\FF_{2^k}$ for $k=2,4$. 

\begin{rem}
Of course, the algorithm expects that $N_1$ is known, for a given
$q$. A different question is finding the $N_1$ points explicitly,
which is fundamental in applications such as in coding theory.
For small $q$, this can often be computed in a straightforward
manner, but otherwise the problem of finding fast algorithms is
quite subtle and appears to be quite involved (cf. Schoof's \cite{Schoof-1995}).
\end{rem}

\section{On the Klein quartic}

The \textit{Klein quartic} $C/\FF_2$ is given by the homogeneous equation
\begin{equation}\label{eq:k4ic}
F(x,y,z)=x^3y+y^3z+z^3x.
\end{equation}
It is is non-singular, absolutely irreducible and has genus 3. 
Let us compute $N_1,N_2,N_3$. 
First notice that the three points (in homogeneous coordinates)
$(1,0,0)$, $(0,1,0)$ and $(0,0,1)$ are the only ones that satisfy
$xyz=0$. In particular, there are two points at infinity.
If $xyz\ne 0$, then we can look at the affine curve $C_z=x^3y+y^3+x$.
Over $\FF_2$ it is clear that there are no more points, hence
$N_1=3$. Over $\FF_4$, there are two more points:
$(\alpha,\alpha^2,1)$ and $(\alpha^2,\alpha,1)$, where
$\alpha^2=\alpha+1$, and so $N_2=5$. To get $N_3$,
let $\FF_8$ be generated by $\beta$ with $\beta^3=\beta+1$.
Since $y^3=y^{10}$, on dividing $C_z$ by $y^3$ we get
$(x/y^3)^3+1+x/y^3=0$. Since $\xi^3+\xi+1=0$ has three solutions
in $\FF_8$ ($\beta, \beta^2,\beta^4$), we conclude that 
$C_z$ has $7\times 3=21$ poins other than $(0,0)$ that are
$\FF_8$-rational and therefore $N_3=24$.
With this, the values for $N_k$ supplied by \textsf{XN}
(for $k\le 12)$ are the following:

\smallskip
\begin{tabular}{c|cccccccccccc}
$k$   &1 &2 &3  &4  &5  &6  &7  &8  &9  &10  &11  &12   \\
\hline
$N_k$ &3 &5 &24 &17 &33 &38 &129&257&528&1025&2049&4238 
\end{tabular}

\medskip
Over $\FF_5$, one finds that $N_1=6$, $N_2=26$ and $N_3=126$. With this,
we can find a similar table (for $k=1,\dotsc,9$):

\smallskip
\begin{tabular}{c|cccccccccc}
$k$   &1 &2 &3  &4  &5  &6  &7  &8  &9    \\
\hline
$N_k$ &6 &26 &126 &626 &3126 &16376 & 78126 & 390626 & 1953126 
\end{tabular}

\appendix

\section{Newton sums}\label{ap:newton-sums}
Let $\seq{\alpha}{1}{n}$ be variables. For $j=1,\dotsc,n$,
let 
\[c_j(\seq{\alpha}{1}{n})=(-1)^j\sigma_j(\seq{\alpha}{1}{n}),
\]
where $\sigma_j$ is the degree $j$ symmetric polynomial in $\seq{\alpha}{1}{n}$.
Finally, let $S_j=\alpha_1^j+\cdots+\alpha_n^j$ for all $j\ge0$, with the convention
$S_0=n$. 

\begin{prop}[Girard-Newton identities]\label{lemma:newton-sums}
\emph{(1)} If $j\ge n$, then 
\begin{equation}\label{eq:rec1}
S_j+c_1S_{j-1}+\cdots+c_{n-1}S_{j-(n-1)}+c_nS_{j-n}=0.
\end{equation}
\emph{(2)} If  $1\le j\le n$, then
\begin{equation}\label{eq:rec2}
S_j+c_1S_{j-1}+\cdots+c_{j-1}S_1+jc_j=0.
\end{equation}
\end{prop}
\begin{proof}
(1) Let $X$ be a new variable. From the definitions it follows that 
\[
\sum_{k=0}^{n}c_k X^{n-k}=(X-\alpha_1)\cdots(X-\alpha_n).
\] 
Therefore 
$\sum_{k=0}^{n}c_k \alpha_i^{n-k}=0$ for $i=1,\dotsc,n$ (Vieta's formulas). 
If we multiply this relation by $\alpha_i^{j-n}$ ($j\ge n$) and sum for $i=1,\dotsc,n$, 
we get $\sum_{k=0}^{n}c_k S_{j-k}=0$, which is the stated equation.

(2) We will proceed by induction on $n$. For $n=1$, the statement,
namely $S_1+c_1=0$, is tautologically true. Assume now that
the statement is true for $n-1$:
\[
S'_{j}+c'_1S'_{j-1}+\cdots+c'_{j-1}S'_1+jc'_j=0
\]
for $1\le j\le n-1$, where the $S'_j$ and $c'_j$ have the same meaning as
$S_j$ and $c_j$, but with respect to the variables 
$\alpha_1,\dotsc,\alpha_{n-1}$. The key observation is that $S'_j$ and $c'_j$
coincide with the result of setting $\alpha_n=0$ in $S_j$ and $c_j$.
This implies that the polynomials
$S_j+c_1S_{j-1}+\cdots+c_{j-1}S_{1}+jc_j$
(for $1\le j\le n-1$) are divisible by $\alpha_n$. Since they are
symmetric in $\seq{\alpha}{1}{n}$, they are divisible by $\alpha_1\cdots\alpha_n$.
Therefore they vanish, as their degrees are $<n$. Finally note
that the equality for $j=n$ has been established in $(1)$. 
\end{proof}

\section{Python code}\label{ap:python-code}

The listings in this appendix are available for downloading at
\href{https://mat-web.upc.edu/people/sebastia.xambo/Listings-MSX}{Listings-MSX}.

\paragraph{The function \textsf{XN}}
This function implements our main algorithm (Section \ref{sec:main})
Besides the standard Python facilities, only 
a working implementation of the rational numbers 
(here denoted \textsf{Q}) is needed.

{\small
\begin{verbatim}
def XN(q,X,k):
    g = len(X)  # the genus of the curve
    if k<=g: return X[:k]   # only k>g gives something new
    X = [0]+X   # trick so that X[j] refers to F_{q^j}
    X = [x>>Q for x in X]  # consider X as a list of Q 
    S = [q**(j)+1-X[j] for j in range(1,g+1)] # First g Newton sums
    S = [0]+S   # similar trick as for X
    # Computation of c_1,...,c_g; set c0=1 
    c = [1>>Q] # to compute in the rational field
    for j in range(1,g+1):
        cj = S[j]
        for i in range(1,j):
            cj += c[i]*S[j-i]
        c += [-cj/j]
    # Add c_{g+i}, for i=1,...,g
    for i in range(1,g+1):
        c += [q**i*c[g-i]]
    # Find Sj for j = g+1,...,k
    for j in range(g+1,k+1):
        if j>2*g:
            Sj=0
        else: 
            Sj = j*c[j]
        for i in range(1,j):
            if i>2*g: break
            Sj += c[i]*S[j-i]
        S += [-Sj]
    # Find X[i] for i = g+1,...,k
    for i in range(g+1,k+1):
        X += [q**i+1-S[i]]
    return X[1:]
\end{verbatim}
}

\paragraph{The Deuring function} This function implements
Deuring's algorithm to list the possible cardinals $\# E(\FF_q)$
of the elliptic curves $E/\FF_q$. Our main reference here has been
\cite{Waterhouse-1969}. We have split the computation in two
parts: the function \textsf{Deuring\_offsets(q)} (which computes
the list of integers $t$ in the segment $[-m,m]$, $m=\floor{2\sqrt{q}}$,
such that $\# E(\FF_q)=q+1-t$ for some $E$), and
\textsf{Deuring\_set(q)}, that outputs the list in question.

\newpage
{\small
\begin{verbatim}
def deuring_offsets(q):
    P = prime_factors(q)  # prime_factors(12) => [2, 2, 3]
    p = P[0]; n = len(P)
    m = int(2*sqrt(q))
    D = [t for t in range(-m,m+1) if gcd(p,t)==1]
    if n%2==0: 
        r = p**(n//2)
        D += [-2*r,2*r] 
        if p%3 != 1: 
            D += [-r,r]
    if n%2 and (p==2 or p==3):
        r = p**((n+1)//2)
        D += [-r,r]
    if n%2 or (n%2==0 and p%4!=1):
        D += [0]
    return sorted([t for t in D])
#
def Deuring_set(q):
    D =deuring_offsets(q)
    return [t+q+1 for t in D]
\end{verbatim}
}

Examples: Here are the Deuring lists for
the first 8 prime powers $q$
{\small
\begin{verbatim}
    q=2  (m=2): [1, 2, 3, 4, 5]
    q=3  (m=3): [1, 2, 3, 4, 5, 6, 7]
    q=4  (m=4): [1, 2, 3, 4, 5, 6, 7, 8, 9]
    q=5  (m=4): [2, 3, 4, 5, 6, 7, 8, 9, 10]
    q=7  (m=5): [3, 4, 5, 6, 7, 8, 9, 10, 11, 12, 13]
    q=8  (m=5): [4, 5, 6, 8, 9, 10, 12, 13, 14]
    q=9  (m=6): [4, 5, 6, 7, 8, 9, 10, 11, 12, 13, 14, 15, 16]
    q=11 (m=6): [6, 7, 8, 9, 10, 11, 12, 13, 14, 15, 16, 17, 18]
\end{verbatim}
}

\paragraph{The Serre procedure for the function $N_q(g)$, $g=1,2,3$}
For the function \textsf{Serre} we have followed \cite{Serre-1983a}.

\begin{verbatim}
def Serre(q, g=1):
    D = ifactor(q)          # ifactor(12) =>  {2: 2, 3: 1}
    if len(D)>1: 
		return 'Serre: {} is not a prime power'.format(q)
    p = list(D)[0]
    e = D[p]                # q = p^e
    m = int(2*sqrt(q))      # g·m is the HWS bound
    if g==1:
        if e%2 and e>=3 and m%p==0: return q+m
        else: return q+m+1
    if g==2:
        if q==4: return 10
        if q==9: return 20
        if e%2==0: return q+1+2*m
        def special(s):
            if m%p==0 or is_square(s-1) or 
                         is_square(4*s-3) or 
                         is_square(4*s-7):
                return True
            else: return False
        if special(q):
            if 2*sqrt(q)-m > (sqrt(5)-1)/2: return q+2*m
            else: return q+2*m-1
        return q+1+2*m
    if g==3:
        P = [2,3,4,5,7,8,9]
        T={2:7,3:10,4:14,5:16,7:20,8:24,9:28}
        if P.count(q): return T[q]
    return "Serre> I do not know the value of N({},{})".format(q,g)
\end{verbatim}


\end{document}